\documentclass{article}

\pdfoutput=1

\usepackage{amsmath,amssymb,amsthm,dirtytalk,comment,float,bm,color,mathrsfs,graphicx,enumitem,natbib}
\RequirePackage[colorlinks,citecolor=blue,urlcolor=blue]{hyperref}

\newtheorem{theorem}{Theorem}[section]
\newtheorem{lemma}[theorem]{Lemma}
\newtheorem{corollary}[theorem]{Corollary}

\theoremstyle{definition}
\newtheorem{example}[theorem]{Example}

\theoremstyle{remark}
\newtheorem{remark}[theorem]{Remark}

\renewcommand{\P}{\mathbb{P}}
\newcommand{\E}{\mathbb{E}}
\newcommand{\F}{\mathscr{F}}
\newcommand{\G}{\mathscr{G}}

\newcommand{\RTP}{\mathrm{RTP}}
\newcommand{\HA}{\mathrm{HA}}

\title{Long-term behavior of casino games}

\author{
S. N. Ethier\thanks{Department of Mathematics, University of Utah. \href{mailto:ethier@math.utah.edu}{ethier@math.utah.edu}} 
\and 
L. Stefanello\thanks{\href{mailto:stefanello.gm@gmail.com}{stefanello.gm@gmail.com}. The author thanks H. A. Mimun for valuable discussions during the early stages of this work.}
}

\date{}

\begin{document}
\maketitle

\begin{abstract}
We study the asymptotic behavior of the ratio of total return (or total profit) to total amount bet in a casino game. While the limit is well understood when the sequence of wagers is independent and identically distributed, here we consider the case in which bet sizes vary over time and may depend on past outcomes. We propose a general framework that yields such results under mild conditions on the conditional expectations of bets, returns, and profits. The set-up applies to many casino games (including compound games and those in which wagers are not immediately resolved), expressing the long-term behavior in terms of intrinsic parameters, namely return to player (RTP) and house advantage (HA). As an application, we examine the roulette win documented in \citeauthor{L76}'s (\citeyear{L76}) \textit{Thirteen against the Bank} and attempt to quantify the likelihood that the story is true.
\end{abstract}

\section{Introduction}

Given a casino game, one of the most natural questions to ask, from the perspective of both the casino and the gambler, is how the total return (or the total profit) compares, in the long term, with the total amount bet.   This ratio not only measures how favorable or unfavorable a game is, but also quantifies to what extent it is in favor of one side or the other (typically, the casino).  In practice, it represents the proportion of the total amount bet that is returned to the gambler (or the gambler's profit as a fraction of the total amount bet).

Because of its importance, this type of question has been studied extensively, and two key parameters are traditionally used to describe it: the \textit{return to player} (RTP) and the \textit{house advantage} (HA).   The former is typically used for games, such as slot machines and video poker, in which payoffs are stated in terms of returns (e.g., 800 for 1 for a royal flush with a max-coin bet in several forms of video poker).  The latter is typically used for games, such as roulette and craps, in which payoffs are stated in terms of profits (e.g., 35 to 1 for winning a single-number bet in roulette).  Note the use of ``for'' (in 800 for 1) versus ``to'' (in 35 to 1).

To encompass both perspectives, we define a wager as a triple $(B, R, X)$ of random variables, where $B$ is the \textit{bet}, $R$ is the \textit{return}, and $X$ is the \textit{profit}.  Here $B$ and $R$ are nonnegative, and the formula relating the three random variables is
\[
    X = R - B,
\]
which implies that $X^- \le B$, where, by definition, $X^- := -\min(X,0)$. It is natural to assume that $R=X=0$ when $B=0$.

Given a sequence $(B_n, R_n, X_n)_{n\ge 1}$ of wagers, we are interested in the long-term behavior of the quantity
\[
 \mathfrak{R}_n := \frac{\sum_{k=1}^n R_k}{\sum_{k=1}^n B_k}
 \quad\text{or}\quad
 \mathfrak{X}_n := \frac{\sum_{k=1}^n X_k}{\sum_{k=1}^n B_k},
\]
that is, the ratio between total returns, or total profits, and total amounts bet.  
The two parameters associated with this study, given a wager $(B, R, X)$, are
\begin{equation}\label{RTP-HA}
    \RTP(B,R) := \frac{\E[R]}{\E[B]}
    \quad\text{and}\quad
    \HA(B,X) := -\frac{\E[X]}{\E[B]},
\end{equation}
defined whenever $0 < \E[B] < \infty$, $\E[R] < \infty$, and $\E[|X|] < \infty$.

The RTP gives the expected proportion of the bet that is returned to the gambler, while the HA expresses the expected proportion kept by the house.   Notice that
\begin{equation*}
\HA(B,X) = -\frac{\E[X]}{\E[B]}
= -\frac{\E[R - B]}{\E[B]}
= 1 - \frac{\E[R]}{\E[B]}
= 1 - \RTP(B,R);
\end{equation*}
this shows that the two approaches, usually followed separately, are actually equivalent.

In \citeauthor{E10} (\citeyear{E10}, Section 6.1), two definitions of the house advantage were given, differing in how ties are accounted for.  Equation~\eqref{RTP-HA} is consistent with the one that includes ties as possible outcomes; the other one regards ties as merely delays in the ultimate resolution of a wager, often described briefly as excluding ties.  For several reasons, it is preferable to include ties.

A first fundamental application of these quantities is that, when the same wager is made repeatedly and independently under identical  conditions, the ratio between total return and total amount bet converges almost surely to the RTP.  Equivalently, the ratio between total profit to the gambler and total amount bet converges almost surely to $-\HA$.  This is an immediate consequence of the strong law of large numbers.  But the underlying assumption of a sequence of i.i.d.\ wagers does not accurately model what happens in a casino.

Therefore, it is natural to relax the i.i.d.\ assumption.  This is precisely the goal of this work: to propose a general framework in which the same type of limit results can be obtained under more realistic assumptions, and in which the asymptotic behavior depends only on the structure of the game, with the gambler’s strategy influencing the limit only within the range permitted by intrinsic parameters.

Following some preliminary considerations in Section~\ref{sec:prelims}, we introduce the main theorem in Section~\ref{sec:main}, formulated in terms of conditional expectations.  We apply it in several scenarios.  First we consider simple games, with one bet (Section~\ref{sec:simple-one}) and multiple bets (Section~\ref{sec:simple-multiple}) per coup.  Bets in simple games depend only on the past and are resolved immediately.  Then we consider compound games, with one bet (Section~\ref{sec:compound-one}) and multiple bets (Section~\ref{sec:compound-multiple}) per coup.  Compound games, which allow additional bets to be made during the course of a coup, have been studied in \cite{D69}, \cite{Ca00}, and \citeauthor{E10} (\citeyear{E10}, Theorem~6.1.8).  The best-known examples are doubling and splitting in blackjack.  Bets in compound games depend on the past and present and are resolved immediately.  Finally, we consider games with future-dependence (Section~\ref{sec:future-dependence}).  Here bets need not be resolved immediately. The canonical example is craps.

We should also mention Section~\ref{sec:Leigh}, which provides a thorough discussion of the roulette system described in \citeauthor{L76}'s (\citeyear{L76}) \textit{Thirteen against the Bank}.  We attempt to quantify the likelihood that Leigh's dramatic story of an 800,000-franc win is true.

We conclude with two remarks.  First, we note that the definition of a wager $(B,R,X)$ is intentionally flexible: it may refer either to the bet, return, and profit of a single gambler, or to the total bet, total return, and total profit at the table (involving multiple gamblers) from the casino’s point of view.  The same theory applies in both interpretations.

In addition, while the general definition of a wager involves three random variables, as already mentioned we can often focus our attention on just two of them, depending on the game considered.

\section{Preliminaries}\label{sec:prelims}

We record here two results that will be used in the proof of the main theorem.  

The first one is the classical result of~\cite{Ch67} concerning the strong law of large numbers for martingales.

\begin{theorem}[Chow]\label{theorem:chow}
Let $(Y_n)_{n\ge 1}$ be a sequence of random variables adapted to a filtration $(\F_n)_{n\ge0}$ such that 
\[
\E[Y_n\mid \F_{n-1}]=0,\quad n\ge1.
\]
Consider the martingale $(S_n)_{n\ge 0}$ with respect to $(\F_n)_{n\ge0}$, where
$S_0:=0$ and $S_n:=Y_1+\cdots+Y_n$ for $n\ge1$.  Suppose that
    \[
        \sum_{k\ge 1} \frac{\E[Y_k^{2}]}{k^{2}} < \infty.
    \]
    Then
    \[
        \lim_{n\to\infty} \frac{S_n}{n} = 0\quad\text{a.s.}
    \]
\end{theorem}

See \cite{Cs68} for an alternative formulation.

The second result needed is a simple but useful technical lemma that will allow us to determine the limiting behavior of ratios of sums when some of the summands become negligible.

\begin{lemma}\label{lemma:technical}
Let $(A_n)_{n\ge1}$ and $(A_n')_{n\ge1}$ be sequences of real numbers, and let $(B_n)_{n\ge1}$ and $(B_n')_{n\ge1}$ be sequences of nonnegative numbers.  Suppose that
\[
    \lim_{n\to\infty} A_n = 0, \quad \lim_{n\to\infty} B_n = 0, \quad\text{and}\quad \liminf_{n\to\infty} B_n' > 0.
\]
Then
\[
    \liminf_{n\to\infty}\frac{A_n'}{B_n'}
    =
    \liminf_{n\to\infty}\frac{A_n+A_n'}{B_n+B_n'}
    \le
    \limsup_{n\to\infty}\frac{A_n+A_n'}{B_n+B_n'}
    =
    \limsup_{n\to\infty}\frac{A_n'}{B_n'}.
\]
\end{lemma}

\begin{remark}
(a) It may happen that $B_n'=0$ for finitely many $n$.

(b) Assuming $(A_n)_{n\ge1}$, $(A_n')_{n\ge1}$, $(B_n)_{n\ge1}$, and $(B_n')_{n\ge1}$ are sequences of random variables, if the hypotheses hold almost surely, then the conclusions hold almost surely.
\end{remark}

\begin{proof}
We prove the conclusion for the $\limsup$, the one for the $\liminf$ being analogous.  
Dividing numerator and denominator by $B_n'$ gives
\[
    \frac{A_n+A_n'}{B_n+B_n'}
    = \frac{\dfrac{A_n}{B_n'}+\dfrac{A_n'}{B_n'}}{\dfrac{B_n}{B_n'}+1}.
\]
Since $B_n' \ge b$ eventually for some $b>0$, we have $A_n/B_n' \to 0$ and $B_n/B_n' \to 0$. Therefore,
\[
    \limsup_{n\to\infty}\frac{A_n+A_n'}{B_n+B_n'}
    =
    \frac{\limsup_{n\to\infty}\bigg(\dfrac{A_n}{B_n'}+\dfrac{A_n'}{B_n'}\bigg)}
         {\lim_{n\to\infty}\bigg(\dfrac{B_n}{B_n'}+1\bigg)}
    = \limsup_{n\to\infty}\frac{A_n'}{B_n'}.
\qedhere
\]
\end{proof}

\section{Main result}\label{sec:main}

In the following, given a sequence of wagers $(B_n,R_n,X_n)_{n\ge1}$, we write
\[
 \mathfrak{R}_n := \frac{\sum_{k=1}^n R_k}{\sum_{k=1}^n B_k}
 \quad\text{and}\quad
 \mathfrak{X}_n := \frac{\sum_{k=1}^n X_k}{\sum_{k=1}^n B_k}
\]
for the random variables whose long-term behavior we are interested in.

We now state our main result. It consists of two parts.  The first, more technical in nature, explains how, in these ratios, we can replace the random variables with their expectations ``with respect to the past,'' given an appropriate filtration.  As these conditional expectations are in general much easier to describe and to evaluate, we can then find explicit answers to the questions we posed in the introduction. 

In the second part, we present explicit bounds on the long-term behavior of wagers, in terms of constants that represent the worst and best strategies for the gambler and, as we will see in the following examples, are determined by intrinsic parameters of the game. 

\begin{theorem}\label{theorem:main}
Let $(B_n,R_n,X_n)_{n\ge1}$ be a sequence of wagers, where $B_n$, $R_n$, and $X_n$ represent the bet, return, and profit at the $n$th coup.  Let $(\F_n)_{n\ge 0}$ be the filtration in which $\F_n$ is the $\sigma$-algebra of events whose occurrence or nonoccurrence is known following the $n$th coup, and assume that $(B_n)_{n\ge1}$, $(R_n)_{n\ge1}$, and $(X_n)_{n\ge1}$ are adapted.  Suppose further that
\[
\sup_{n\ge1}\E[B_n^2]< \infty, \qquad
\liminf_{n\to\infty}\frac{1}{n}\sum_{k=1}^n B_k>0\quad\text{a.s.}
\]
and $\sup_{n\ge1}\E[R_n^2]<\infty$, or equivalently, $\sup_{n\ge1}\E[X_n^2]<\infty$.

(a) Then 
\begin{align*}
\liminf_{n\to\infty}\frac{\sum_{k=1}^n \E[R_k\mid\F_{k-1}]}{\sum_{k=1}^n \E[B_k\mid\F_{k-1}]}
&=\liminf_{n\to\infty} \mathfrak{R}_n \\
&\le\limsup_{n\to\infty} \mathfrak{R}_n
=\limsup_{n\to\infty}\frac{\sum_{k=1}^n \E[R_k\mid\F_{k-1}]}{\sum_{k=1}^n \E[B_k\mid\F_{k-1}]}
\quad\text{a.s.}
\end{align*}
and 
\begin{align*}
\liminf_{n\to\infty}\frac{\sum_{k=1}^n \E[X_k\mid\F_{k-1}]}{\sum_{k=1}^n \E[B_k\mid\F_{k-1}]}
&=\liminf_{n\to\infty} \mathfrak{X}_n \\
&\le\limsup_{n\to\infty} \mathfrak{X}_n
=\limsup_{n\to\infty}\frac{\sum_{k=1}^n \E[X_k\mid\F_{k-1}]}{\sum_{k=1}^n \E[B_k\mid\F_{k-1}]}
\quad\text{a.s.}
\end{align*}

(b) Moreover, if there exist constants $\rho'$ and $\rho''$ such that\footnote{\label{note1}Here and in the following sections, we regard these inequalities as satisfied if the fractions between them are of the form 0/0.}
\[
    \rho'
    \le \frac{\E[R_n \mid \F_{n-1}]}{\E[B_n \mid \F_{n-1}]}
    \le \rho'' \quad \text{a.s.}
\]
for all $n\ge1$, or equivalently, constants $\chi'$ and $\chi''$ such that 
\[
    \chi'
    \le \frac{\E[X_n \mid \F_{n-1}]}{\E[B_n \mid \F_{n-1}]}
    \le \chi'' \quad \text{a.s.}
\]
for all $n\ge1$, then
\[    \rho' \le \liminf_{n\to\infty}\mathfrak{R}_n
    \le \limsup_{n\to\infty}\mathfrak{R}_n
    \le \rho''\quad \text{a.s.}
\]
and
\[
    \chi' \le \liminf_{n\to\infty}\mathfrak{X}_n
    \le \limsup_{n\to\infty}\mathfrak{X}_n
    \le \chi''
    \quad \text{a.s.}
\]
In particular, if $\rho'=\rho''=\rho$, or equivalently $\chi'=\chi''=\chi$, then 
\[
\lim_{n\to\infty}\mathfrak{R}_n = \rho\quad a.s.\qquad \text{and} \qquad \lim_{n\to\infty}\mathfrak{X}_n = \chi \quad \text{a.s.}
\]
\end{theorem}

\begin{proof}
The starting point is Doob’s decomposition.  Let $(Y_n)_{n\ge1}$ play the role of $(B_n)_{n\ge1}$, $(R_n)_{n\ge1}$, or $(X_n)_{n\ge1}$.  Then
\[
    \sum_{k=1}^n Y_k
    = \sum_{k=1}^n \bigl(Y_k - \E[Y_k \mid \F_{k-1}]\bigr)
      + \sum_{k=1}^n \E[Y_k \mid \F_{k-1}],
\]
and
\[
    \E\bigl[(Y_k - \E[Y_k \mid \F_{k-1}])^2\mid \F_{k-1}\bigr]
    \le \E[Y_k^2\mid \F_{k-1}].
\]
Hence
\[
\sup_{k\ge1} \E\bigl[(Y_k - \E[Y_k \mid \F_{k-1}])^2\bigr]
\le\sup_{k\ge1} \E[Y_k^2]<\infty.
\]
By Theorem~\ref{theorem:chow}, we deduce that
\[
    \lim_{n\to\infty}\frac{1}{n}\sum_{k=1}^n
    \bigl(Y_k - \E[Y_k \mid \F_{k-1}]\bigr) = 0\quad\text{a.s.}
\]

Applying this with $Y_k = B_k$, we obtain
\[
\frac{1}{n}\sum_{k=1}^n \E[B_k \mid \F_{k-1}]=\frac{1}{n}\sum_{k=1}^n B_k-\frac{1}{n}\sum_{k=1}^n \bigl(B_k - \E[B_k \mid \F_{k-1}]\bigr),
\]
so the assumption
\[
\liminf_{n\to\infty}\frac{1}{n}\sum_{k=1}^n B_k>0\quad\text{a.s.}
\]
implies
\[
\liminf_{n\to\infty}\frac{1}{n}\sum_{k=1}^n \E[B_k \mid \F_{k-1}]>0\quad\text{a.s.}
\]

For the return ratio, we can therefore write
\[
    \mathfrak{R}_n = \frac{\frac{1}{n}\sum_{k=1}^n (R_k - \E[R_k \mid \F_{k-1}])
          + \frac{1}{n}\sum_{k=1}^n \E[R_k \mid \F_{k-1}]}{\frac{1}{n}\sum_{k=1}^n (B_k - \E[B_k \mid \F_{k-1}])
          + \frac{1}{n}\sum_{k=1}^n \E[B_k \mid \F_{k-1}]}
\]
and apply Lemma~\ref{lemma:technical} with
\[
A_n := \frac{1}{n}\sum_{k=1}^n (R_k - \E[R_k \mid \F_{k-1}]),
\quad
A_n' := \frac{1}{n}\sum_{k=1}^n \E[R_k \mid \F_{k-1}],
\]
\[
B_n := \frac{1}{n}\sum_{k=1}^n (B_k - \E[B_k \mid \F_{k-1}]),
\quad
B_n' := \frac{1}{n}\sum_{k=1}^n \E[B_k \mid \F_{k-1}],
\]
to obtain the stated equalities and inequalities for $(\mathfrak{R}_n)_{n\ge1}$.  
The argument for $(\mathfrak{X}_n)_{n\ge1}$ is identical, with $R_k$ replaced by $X_k$.
\end{proof}

We conclude this section with a few comments concerning the assumptions of Theorem~\ref{theorem:main}.

\begin{remark}
(a) The filtration is indexed by coups for convenience; it could just as well be indexed by spins, rolls, rounds, etc.  
Excluding compound games and games with future dependence, bets are typically \textit{predictable}, so that $\E[B_n \mid \F_{n-1}] = B_n$ for every $n\ge1$. 

(b) The requirement
    \[
        \liminf_{n\to\infty} \frac{1}{n}\sum_{k=1}^n B_k > 0\quad\text{a.s.}
    \]
    ensures that the average amount wagered (averaged over time) does not vanish in the long term.  
    This is automatically satisfied if, for instance, a table minimum $\delta>0$ is imposed, which is standard in real casino play, and the bettor makes a bet at every coup.  One may also allow occasional zero bets, provided that the nonzero bets have positive density. 

(c) The assumptions on second moments,
    \[
\sup_{n \ge 1} \E[B_n^2] < \infty, \quad \sup_{n \ge 1} \E[R_n^2] < \infty,\quad\text{and}\quad\sup_{n \ge 1} \E[R_n^2] < \infty,
    \]
permit application of Theorem~\ref{theorem:chow}.
    In any real casino, bets, returns, and profits are uniformly bounded, so this condition is trivially satisfied.  
    A single wager may aggregate a random number of underlying components, so that $B_n$, $R_n$, and $X_n$ need not be uniformly bounded, yet their second moments remain bounded.
\end{remark}

\section{Simple games with one bet per coup}\label{sec:simple-one}

We begin with the case of \textit{simple} games, in which at each coup $n$ the gambler chooses a bet $B_n$, and a random outcome is observed.  
This outcome determines a return $R_n^{\circ}$ per unit bet and a profit $X_n^{\circ}$ per unit bet, so that the total return and profit of the wager are
\begin{equation*}
    R_n = B_n R_n^{\circ}
    \qquad\text{and}\qquad
    X_n = B_n X_n^{\circ}.
\end{equation*}
Thus $(B_n, R_n, X_n)$ is the wager placed at coup $n$.  
Let $(\F_n)_{n\ge 0}$ be the natural filtration, where $\F_n$ comprises all events whose occurrence or nonoccurrence is known following the $n$th coup.

In this simple setting,
\begin{itemize}
    \item the sequence $(B_n)_{n\ge1}$ is predictable with respect to $(\F_n)_{n\ge0}$;
    \item the sequence $(R_n^\circ)_{n\ge1}$ is i.i.d.\ and adapted to $(\F_n)_{n\ge0}$, and $R_n^\circ$ is independent of $\F_{n-1}$ for each $n\ge1$;
    \item the sequence $(X_n^\circ)_{n\ge1}$ is i.i.d.\ and adapted to $(\F_n)_{n\ge0}$, and $X_n^\circ$ is independent of $\F_{n-1}$ for each $n\ge1$.
\end{itemize}

We immediately obtain the following consequence of Theorem~\ref{theorem:main}.

\begin{corollary}\label{cor:basic}
    Under the above assumptions, suppose additionally that
\begin{equation}\label{bet-conditions}
\sup_{n \ge 1} \E[B_n^2] < \infty
\qquad\text{and}\qquad
\liminf_{n\to\infty}\frac{1}{n}\sum_{k=1}^n B_k>0\quad\text{a.s.}
\end{equation}

 If $\E[(R_1^\circ)^2] < \infty$, or equivalently, $\E[(X_1^\circ)^2] < \infty$, then
\[
\lim_{n\to\infty}\mathfrak{R}_n=\RTP(1,R_1^\circ)\quad\text{a.s.} \qquad \text{and} \qquad \lim_{n\to\infty}\mathfrak{X}_n=-\HA(1,X_1^\circ)\quad\text{a.s.}
\]
\end{corollary}

\begin{remark}
Most modern games developed in the online casino industry naturally fit into the model described in this section. In particular,  every wager shares the same RTP and house advantage. This explains why, outside academic treatments, one typically refers to \say{the RTP and the HA of a game} rather than of individual wagers.
\end{remark}

\begin{remark}
There are several results closely related to Corollary 4.1 in the mathematical literature. For example, \cite{LW95} establish an almost-sure limit theorem for selection systems (those in which the bet sizes are \{0,1\}-valued). A result for more-general betting systems by the same authors appears in \cite{LW03} (see their Corollary 6), which under an i.i.d. assumption on the profits per unit bet reduces to a result effectively equivalent to our Corollary 4.1. We note that their approach does not rely on martingale methods.
\end{remark}

The following two examples give some insight into the assumed second-moment bound on the bet sizes.

\begin{example}[Martingale system without a betting limit]
Let $(X_n^\circ)_{n\ge1}$ be i.i.d.\ with
\[
\P(X_1^\circ=1)=p\quad\text{and}\quad
\P(X_1^\circ=-1)=1-p,\quad\text{where}\quad
0<p<\frac12,
\]
so that the underlying even-money proposition favors the house.  
Assume there is no house betting limit and define $B_1:=1$ and
\[
B_n:=
\begin{cases}
2B_{n-1}, &\text{if $X_{n-1}^\circ=-1$},\\[2mm]
1,        &\text{otherwise},
\end{cases}
\qquad n\ge2.
\]
The bettor begins with a bet of one unit, then doubles the stake after each loss and returns to a one-unit stake after each win.

Here $\sup_{n\ge1} \E[B_n^2]=\infty$ because $\E[B_n^2]\ge(2^{n-1})^2(1-p)^{n-1}\ge2^{n-1}$ for each $n\ge1$. 
Furthermore, the conclusion of Corollary~\ref{cor:basic} fails: although the house advantage is positive ($1-2p>0$), the bettor's cumulative profit becomes positive after each win.
\end{example}

\begin{example}[Martingale system with a betting limit]
Let $(X_n^\circ)_{n\ge1}$ be as in the previous example, and let $M\ge1$ be the house betting limit.  
Define $B_1:=1$ and
\[
B_n:=
\begin{cases}
2B_{n-1}, &\text{if $X_{n-1}^\circ=-1$ and $2B_{n-1}\le M$},\\[2mm]
1,        &\text{otherwise},
\end{cases}
\qquad n\ge2.
\]
Whenever the prescribed doubling would exceed the betting limit, the bettor returns to a one-unit stake.  
All bets are therefore bounded by $M$, so $\sup_{n \ge 1} \E[B_n^2] < \infty$, and the assumptions of Corollary~\ref{cor:basic} are satisfied.  
Hence the long-term ratio of total profit to total bet converges to minus the house advantage of the game, $2p-1$, almost surely.
\end{example}

\section{Simple games with multiple bets per coup}\label{sec:simple-multiple}

We continue to consider simple games, but now with $d\ge1$ distinct betting opportunities at each coup.  
Each opportunity behaves like an individual simple game, and the overall wager is obtained by combining these $d$ components.

The bet is now a vector
\[
\bm B_n = (B_{n,1},\ldots,B_{n,d}),
\]
and the returns per unit bet and profits per unit bet are vectors
\[
\bm R_n^{\circ}=(R_{n,1}^{\circ},\ldots,R_{n,d}^{\circ})
\quad \text{and}\quad
\bm X_n^{\circ}=(X_{n,1}^{\circ},\ldots,X_{n,d}^{\circ}).
\]
The total bet, total return, and total profit for coup $n$ are
\[
B_n=\bm B_n\bm\cdot\bm1, \quad R_n = \bm B_n \bm\cdot \bm R_n^{\circ},
\quad\text{and}\quad
X_n = \bm B_n\bm\cdot \bm X_n^{\circ}.
\]
Let $(\F_n)_{n\ge 0}$ be the natural filtration.

In this setting,
\begin{itemize}
    \item the sequence $(\bm B_n)_{n\ge1}$ is predictable with respect to $(\F_n)_{n\ge0}$;
    \item the sequence $(\bm R_n^\circ)_{n\ge1}$ is i.i.d.\ and adapted to $(\F_n)_{n\ge0}$, and $\bm R_n^\circ$ is independent of $\F_{n-1}$ for each $n\ge1$;
    \item the sequence $(\bm X_n^\circ)_{n\ge1}$ is i.i.d.\ and adapted to $(\F_n)_{n\ge0}$, and $\bm X_n^\circ$ is independent of $\F_{n-1}$ for each $n\ge1$.
\end{itemize}

Under these assumptions,
\begin{align*}
\frac{\sum_{k=1}^n \E[R_k\mid\F_{k-1}]}{\sum_{k=1}^n \E[B_k\mid\F_{k-1}]}&=\frac{\sum_{k=1}^n \bm B_k\bm\cdot\E[\bm R_1^\circ]}{\sum_{k=1}^n \bm B_k\bm\cdot\bm1},\\
\frac{\sum_{k=1}^n \E[X_k\mid\F_{k-1}]}{\sum_{k=1}^n \E[B_k\mid\F_{k-1}]}&=\frac{\sum_{k=1}^n \bm B_k\bm\cdot\E[\bm X_1^\circ]}{\sum_{k=1}^n \bm B_k\bm\cdot\bm1}.
\end{align*}

\begin{corollary}\label{cor:multiple_bets}
Under the above assumptions, suppose also that 
\[
\sup_{n \ge 1} \E[|\bm B_n|^2] < \infty\qquad\text{and}\qquad\liminf_{n\to\infty}\frac{1}{n}\sum_{k=1}^n \bm B_k\bm\cdot\bm1>0\quad\text{a.s.}
\]
If $\E[|\bm R_1^\circ|^2] < \infty$, or equivalently, $\E[|\bm X_1^\circ|^2] < \infty$, then
\begin{align*}
&\min_{1\le i \le d}\RTP(1,R_{1,i}^\circ)\le\liminf_{n\to\infty}\frac{\sum_{k=1}^n \bm B_k\bm\cdot\E[\bm R_1^\circ]}{\sum_{k=1}^n \bm B_k\bm\cdot\bm1}=\liminf_{n\to\infty}\mathfrak{R}_n\\
&\qquad\le\limsup_{n\to\infty}\mathfrak{R}_n=\limsup_{n\to\infty}\frac{\sum_{k=1}^n \bm B_k\bm\cdot\E[\bm R_1^\circ]}{\sum_{k=1}^n \bm B_k\bm\cdot\bm1}\le\max_{1\le i \le d}\RTP(1,R_{1,i}^\circ)\quad\text{a.s.}
\end{align*}
and
\begin{align*}
&-\max_{1\le i \le d}\HA(1,X_{1,i}^\circ)\le\liminf_{n\to\infty}\frac{\sum_{k=1}^n \bm B_k\bm\cdot\E[\bm X_1^\circ]}{\sum_{k=1}^n \bm B_k\bm\cdot\bm1}=\liminf_{n\to\infty}\mathfrak{X}_n\\
&\qquad\le\limsup_{n\to\infty}\mathfrak{X}_n=\limsup_{n\to\infty}\frac{\sum_{k=1}^n \bm B_k\bm\cdot\E[\bm X_1^\circ]}{\sum_{k=1}^n \bm B_k\bm\cdot\bm1}\le-\min_{1\le i \le d}\HA(1,X_{1,i}^\circ)\quad\text{a.s.}
\end{align*}
The limits exist under either of two additional conditions: 
\begin{itemize}
    \item If $\RTP(1,R_{1,1}^\circ)=\cdots=\RTP(1,R_{1,d}^\circ)$, or equivalently, $\HA(1,X_{1,1}^\circ)=\cdots=\HA(1,X_{1,d}^\circ)$, then
\[
\lim_{n\to\infty}\mathfrak{R}_n=\RTP(1,R_{1,1}^\circ)\quad\text{a.s.}
\]
and 
\[
\lim_{n\to\infty}\mathfrak{X}_n=-\HA(1,X_{1,1}^\circ)\quad\text{a.s.}
\]
\item If $\lim_{n\to\infty}\sum_{k=1}^n \bm B_k/\sum_{k=1}^n \bm B_k\bm\cdot\bm1=\bm w$ a.s., then
\[
\lim_{n\to\infty}\mathfrak{R}_n=\sum_{i=1}^d w_i\,\RTP(1,R_{1,i}^\circ)\quad\text{a.s.}
\]
and
\[
\lim_{n\to\infty}\mathfrak{X}_n=-\sum_{i=1}^d w_i\,\HA(1,X_{1,i}^\circ)\quad\text{a.s.}
\]
\end{itemize}
\end{corollary}

\begin{proof}
The proof is immediate, except for the second-moment bound
\begin{align*}
\E[(R_k - \E[R_k \mid \F_{k-1}])^2\mid \F_{k-1}] 
    &\le \E[R_k^2\mid \F_{k-1}]\\
    &=\E[(\bm B_k\bm\cdot\bm R_k^\circ)^2\mid\F_{k-1}]\\
    &\le \E[|\bm B_k|^2|\bm R_k^\circ|^2\mid\F_{k-1}]\\
    &=|\bm B_k|^2\,\E[|\bm R_1^\circ|^2],
\end{align*}
which uses the Cauchy--Schwarz inequality. \qedhere
\end{proof}

\begin{example}[Single-zero roulette]
Each standard betting opportunity in roulette with a single 0 has house advantage $1/37$: the payoff odds are $36/m - 1$ to 1 when $m$ numbers are covered, and the success probability is $m/37$.  
Thus, the game is HA-invariant.  
Regardless of how the gambler distributes bets across opportunities, the long-term profit-to-bet ratio satisfies
\[
\lim_{n\to\infty}\mathfrak{X}_n = -\frac{1}{37}\quad\text{a.s.}
\]
\end{example}

\begin{example}[Single-zero roulette with \textit{partager}]
Under the \textit{partager} rule, even-money bets return half the stake if zero appears.  
This yields a house advantage of $1/74$ for even-money bets and $1/37$ for all other bets.  
The game is therefore not HA-invariant.

Corollary~\ref{cor:multiple_bets} gives
\[
    -\frac{1}{37}
   \le\liminf_{n\to\infty}\mathfrak{X}_n\le
     \limsup_{n\to\infty}\mathfrak{X}_n\le
    -\frac{1}{74}\quad\text{a.s.},
\]
and the precise limit, if it exists, depends on the gambler’s long-term allocation to even-money bets versus other betting opportunities.
\end{example}

\begin{example}[Slot machines and multiple paylines]
Slot machines often allow the gambler to choose how many paylines to activate before each spin.  
Each activated line behaves as a distinct betting opportunity with its own return distribution, and its own RTP, say $\rho$, which is constant among paylines. 

Regardless of how the number of active lines varies from spin to spin, we obtain
\[
\lim_{n\to\infty}\mathfrak{R}_n=\rho\quad\text{a.s.}
\]
\end{example}

\section{Interlude: The Leigh roulette system}\label{sec:Leigh}

In 1966 a team of 13 English men and women set out to beat the roulette tables at the Casino Municipale in Nice, France, using a betting system known as the reverse Labouchere.  They won 800,000 francs (US\$163,000) in eight days and were barred from further play.  Their story is documented in a book (said to be nonfiction) by Norman \cite{L76} titled \textit{Thirteen against the Bank}. 

We use this incident to illustrate Corollary~\ref{cor:multiple_bets} concerning the asymptotic ratio of total profit to total amount bet in a sequence of coups with multiple bets per coup.  We point out that a careful reading of Leigh's book reveals that his claim of an 800,000-franc win is deceptive.   We further assess the likelihood that such a win could have occurred by chance, using Monte Carlo simulation and a Poisson approximation. 

We are of course dealing with a French roulette wheel with a single zero and the \textit{en prison} rule (\citeauthor{L76}, \citeyear{L76}, p.~30\,/\,p.~35),\footnote{The first page number refers to the British edition and the second to the American edition.} which applies to the six even chances at roulette, namely red, black, odd, even, low, and high (or rouge, noir, impair, pair, manque, and passe).  There are several implementations of this rule in the gambling literature.  The one we adopt is simple and natural.  Following a zero spin, the next nonzero spin determines the outcome, a tie if the even chance on which the gambler bet is hit, a loss otherwise.

Since we are concerned in this section only with the Leigh system, we model roulette as a game with six betting opportunities, and for each of them the profit per unit bet has the distribution
\begin{equation}\label{Xcirc-dist}
X_1^\circ=\begin{cases}
1 & \text{with probability $\tfrac{18}{37}$},\\[4pt]
0 & \text{with probability $\tfrac{1}{2}\cdot\tfrac{1}{37}$},\\[4pt]
-1 & \text{with probability $\tfrac{18}{37}+\tfrac{1}{2}\cdot\tfrac{1}{37}$}.
\end{cases}
\end{equation}
The mean is $-1/74$, hence the house advantage is $1/74$, or about 1.351\%.

We model the game as a sequence of independent and identically distributed coups.  A coup comprises a single spin if it is nonzero.  If a zero appears, the next nonzero spin completes the coup.  (Thus, a coup comprises 37/36 spins on average.)  While awaiting resolution of a coup, no further bets are made.  This model accommodates the Leigh system but not an arbitrary roulette system.  

To describe the Leigh system in more detail, we recall the betting system known as the \textit{reverse Labouchere}.  The Labouchere system (direct and reverse) has been the subject of several mathematical studies since the Leigh book appeared (\citeauthor{D80}, \citeyear{D80}; \citeauthor{E08}, \citeyear{E08}; \citeauthor{HW19}, \citeyear{HW19}).  It applies to any even-money proposition, including any of the six even chances at roulette.  The gambler keeps a scorecard with a list of positive numbers that is updated after each coup.  The initial list can be arbitrary, but a popular choice (and the choice of the Leigh team) is $1,2,3,4$, which we will assume hereafter.  Given such a list, the next bet is always the sum of the first and last terms of the list (unless there is only one term, in which case it is just that term).  After each coup, the list is updated as follows:
\begin{itemize}
\item After a win, the amount just won (or, equivalently, the amount just bet) is appended to the list as a new last term.
\item After a loss, the first and last terms are canceled.
\item After a tie, there is no change to the list.
\end{itemize}
The system is aborted if the list becomes empty or if a bet exceeding the house maximum of $M$ units is called for.  
In either case, the system is instantly restarted with $1,2,3,4$ and a bet of 5.

To understand better what happens when the reverse Labouchere betting system is aborted, consider a specified even chance.  Let $X_1^\circ,X_2^\circ,\ldots$ be i.i.d.\ as in \eqref{Xcirc-dist}, with $X_n^\circ$ representing the profit per unit bet on that chance at the $n$th coup.  Let $B_n$ be the amount bet on that chance at the $n$th coup, let $F_n$ be the gambler's cumulative profit from that chance after the $n$th coup, and let $S_n$ be the sum of the terms on the gambler's list for that chance after the $n$th coup.  Then
\[
F_n-F_{n-1}=B_n X_n^\circ\quad\text{and}\quad S_n-S_{n-1}=B_n X_n^\circ,
\]
Thus, $F_n-S_n=F_{n-1}-S_{n-1}=\cdots=F_0-S_0=-10$ since $F_0=0$ and $S_0=10$, and we have
\begin{equation}\label{cum-profit}
F_n=S_n-10,\qquad 0\le n\le T, 
\end{equation}
where $T$ is the first time the system is aborted (either because the list becomes empty or because the system calls for a bet greater than $M$).  Therefore, the gambler can expect frequent small losses ($F_T=-10$ if $S_T=0$) and rare large wins ($F_T=S_T-10$ if the next bet called for is greater than $M$, which implies $S_T>M$).

Leigh’s team had six gamblers simultaneously playing reverse Laboucheres on the same wheel, one on each of the six even chances.  The house betting limits were a minimum of 5 francs and a maximum of 2600 francs, so we take the franc as our betting unit and $M=2600$.  The six gamblers would play for six hours, then be replaced by six other gamblers for the next six hours (the casino operated 12 hours per day, from 3 p.m.\ to 3 a.m.)

A minor technical issue arises in that bets of size 3 or 4, which are below the house minimum, may be called for.  Although Leigh does not address this directly, he does so indirectly when suggesting how to proceed when London's Regency, which the Leigh team used for a practice session, raised the minimum betting limit from two to ten shillings (\citeauthor{L76}, \citeyear{L76}, p.~115\,/\,p.~140).  His solution is to require that, in such cases, the gambler does not bet, but still regards the bet as 3 or 4 for the purpose of the system.  For example, if the first three coups result in a win followed by two losses, the list, starting at $1,2,3,4$, goes to $1,2,3,4,5$, then to $2,3,4$, and then to 3.  The gambler then bets 0 (not 3).  If he ``wins,'' the new list is $3,3$ (not $3,0$) and a bet of 6 follows.  If he ``loses,'' the list becomes empty and the system restarts with $1,2,3,4$ and a bet of 5.

Corollary~\ref{cor:multiple_bets} tells us that the long-term ratio of total profit (for the whole team) to total amount bet converges almost surely to $-1/74$ as the number of coups tends to infinity.  This implies that the system must fail in the long term.  In particular, it cannot be a consistent winner.

But in Leigh's scenario, the number of coups was about 30 per hour, or roughly 360 per day, and thus about 2880 over eight days.  This does not qualify as the ``long term.''  Could the Leigh team have had a lucky run, achieving their success by chance?  To shed some light on this, we first need to be more precise about what the Leigh team achieved.  Table~\ref{tab:Leigh-data} does this.  As we noted, the reverse Labouchere bettor sees frequent small losses equal to 10 and rare large wins equal to the sum of the terms on the gambler's list at the time the system is aborted, less 10 (cf.~\eqref{cum-profit}).  The table records the large wins (called winning progressions), of which there were 27, and disregards the many small losses (called losing progressions).  Thus, the 800,000-franc win was the sum of the wins from the 27 winning progressions, but this was not the profit because the sum of the losses from the losing progressions was not accounted for.  

One might object that maybe the reported wins included the many small losses and can therefore be regarded as profits.  But this can be seen to be false from Leigh's own words (\citeauthor{L76}, \citeyear{L76}, p.~151\,/\,p.~185).  Referring to Mr.~Hopplewell's win on day 3, he wrote,

\begin{table}[H]
\caption{The winning progressions during the Leigh team's eight days at Casino Municipale, as documented in \textit{Thirteen against the Bank}.\label{tab:Leigh-data}}
\catcode`@=\active\def@{\phantom{0}}
\begin{center}
\begin{tabular}{cccccc}
day & no. & bettor & even   & amount won & page reference \\
    &  &        & chance & (in francs)  & (Brit.\,/\,Amer.~ed.) \\
\noalign{\smallskip}\hline\noalign{\smallskip}
1 & @1 & Mrs.\ Heppenstall & even & @49,000 & p.~22\,/\,p.~25 \\
  & @2 & Mr.\ Hopplewell   & black & @29,000 & p.~22\,/\,p.~25 \\
  & & subtotal           &       & @78,000 & p.~23\,/\,p.~26 \\
\noalign{\smallskip}\hline\noalign{\smallskip}
2 & @3 & Mr.\ Robinson & black & @57,950 & p.~139\,/\,p.~171 \\
  & @4 & Mrs.\ Richardson & red & @39,375 & p.~141\,/\,p.~173 \\
  & @5 & Mr.\ Hopplewell & even & @33,125 & p.~142\,/\,p.~175 \\
  & & subtotal         &      & 130,450 & p.~142\,/\,p.~175 \\
\noalign{\smallskip}\hline\noalign{\smallskip}
3 & @6 & Mr.\ Hopplewell &  black & @10,735 & p.~151\,/\,p.~185 \\
  & @7 & Mr.\ Milton & high & @21,680 & p.~151\,/\,p.~186 \\
  & @8 & Mrs.\ Richardson & red & @33,750 & p.~152\,/\,p.~187 \\
  & & subtotal & & @66,165 & p.~155\,/\,p.~190 \\
\noalign{\smallskip}\hline\noalign{\smallskip}
4 & @9 & Mr.\ Nathan & low & @40,300 & p.~159\,/\,p.~196 \\
  & 10 & Mr.\ Milton & high & @26,700 & p.~160\,/\,p.~197 \\
  & & subtotal & & @67,000 & p.~161\,/\,p.~198 \\
\noalign{\smallskip}\hline\noalign{\smallskip} 
5 & 11 & Mrs.\ Harper-Biggs & odd & ? & p.~170\,/\,p.~209 \\
  & 12 & Mrs.\ Richardson & red & ? & p.~171\,/\,p.~210 \\
  & & subtotal &  & @55,385 & p.~172\,/\,p.~211 \\
\noalign{\smallskip}\hline\noalign{\smallskip}   
6 & 13 & Mr.\ Nathan & low & @27,475 & p.~173\,/\,p.~213 \\
  & 14 & Mr.\ Milton & high & @32,000 & p.~175\,/\,p.~215 \\
  & 15 & Mrs.\ Heppenstall & even & @11,768 & p.~176\,/\,p.~217 \\
  & 16 & Mr.\ Fredericks & black & @51,760 & p.~177\,/\,p.~218 \\
  & & subtotal & & 123,003 & p.~177\,/\,p.~218 \\
\noalign{\smallskip}\hline\noalign{\smallskip}  
7 & 17 & Mrs.\ Harper-Biggs & odd & @15,645 & p.~180\,/\,p.~222 \\
  & 18 & Mrs.\ Harper-Biggs & odd & ? & p.~181\,/\,p.~223 \\
  & 19 & Mr.\ Leigh & odd & ? & p.~181\,/\,p.~224 \\ 
  & 20 & Mr.\ Vincent & low & @48,640 & p.~182\,/\,p.~225 \\
  & 21 & Mr.\ Fredericks & black & @12,895 & p.~182\,/\,p.~225 \\
  & 22 & Mr.\ Fredericks & black & @26,500 & p.~183\,/\,p.~226 \\
  & & subtotal &  & 159,660 & p.~183\,/\,p.~226 \\
\noalign{\smallskip}\hline\noalign{\smallskip} 
8 & 23 & Mr.\ Nathan & low & \smash{\raisebox{-5pt}{@37,000}} & \smash{\raisebox{-5pt}{p.~186\,/\,p.~230}} \\
  & 24 & Mr.\ Hopplewell & black &  &  \\
  & 25 & Mr.\ Hopplewell & black & @18,390 & p.~187\,/\,p.~231 \\
  & 26 & Mr.\ Robinson &  ?   & @47,665 & p.~187\,/\,p.~232 \\
  & 27 & Mr.\ Sherlock & even & @16,540 & p.~188\,/\,p.~232 \\
  & & subtotal       &      & 119,595 & p.~188\,/\,p.~232 \\
\noalign{\smallskip}\hline\noalign{\smallskip}
  & & total          &      & 799,258 & \\
\end{tabular}
\end{center}
\end{table}

\begin{quotation}
\noindent Hopplewell lost again. His line now read 526, 633, 740, 847, 978, 1109, 1240, 1399, 1554, 1709.  His next stake was, therefore, $526+1709=2235$.

This time he won.  His next stake would have been $526+2235=2761$, over the limit.  The little adverse sequence at the end had more or less halved his winnings, which were, if you care to add up the last list of figures, 10,735 francs.
\end{quotation}

Here Leigh carelessly neglected to add the 2235 term to the list (from the win at the start of the second quoted paragraph), whose sum should be $10{,}735+2235=12{,}970$.  (He also neglected to subtract 10 from the sum, as in \eqref{cum-profit}, but this has little effect.)  In any case, the main point is that the many small losses were completely neglected, perhaps to make the win seem more impressive than it really was.

To assess the likelihood of Leigh's story, we focus on his claim that there were 27 winning progressions.  We simulated eight days of 360 coups per day, replicating the experiment one million times (hence 2.88 billion coups).  The program, written in \textit{Mathematica}, is straightforward.  We note that an algorithm for determining whether a number in the range 1--36 is red or black is given in \citeauthor{E10} (\citeyear{E10}, p.~462, footnote 1).

Figure~\ref{fig:histogram}
provides a histogram of the ratio of total profit to total amount bet (over the eight-day session), showing a considerable amount of variation.  Moreover, a proportion 0.267993 of the sessions were profitable.  Other relevant statistics from the simulation are shown in Table~\ref{tab:Leigh-results}.

Progressions are either winning, losing, or incomplete.  An incomplete progression usually occurs for each of the even chances at the end of each day (but occasionally a progression is completed precisely at coup 360).  The simulated numbers from Table~\ref{tab:Leigh-results} are consistent with Corollary~\ref{cor:multiple_bets} since
\begin{equation}\label{consistency}
\frac{12{,}227.812000-20{,}502.704884+2691.843352}{413{,}287.742596}\approx-0.0135089,
\end{equation}
with the theoretical limit being $-0.0135135\cdots$. 

There are several ways to see that the Leigh team's reported results are anomalous.
First, the largest number of winning progressions in our one million simulated replications of the Leigh experiment was 9, only one-third of the 27 Leigh reported.  Second, the total amount won during the winning progressions, in the one million simulated replications, was never more than 102,000 francs, much less than the 800,000 francs Leigh reported.  Third, the average amount won during a single winning progression, during our one million simulated replications, was 8236.07 francs, which Leigh's team exceeded all 27 times.  Thus, for several related reasons, we can say that the chance that the Leigh experiment occurred as stated is less that one in a million, but this greatly understates how unusual Leigh's results were.

\begin{figure}[htb]
\begin{center}
\includegraphics[width=4in]{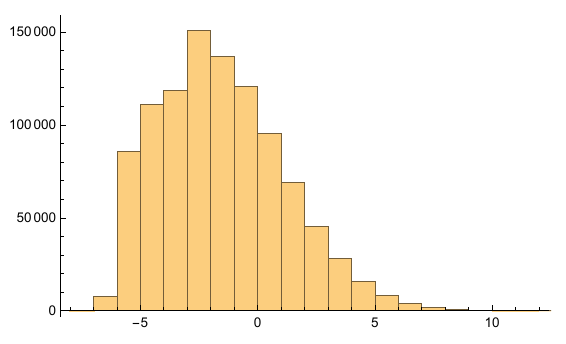}
\end{center}
\caption{\label{fig:histogram}Histogram for the ratio of total profit to total amount bet (in percentage terms) in eight 360-coup days of the Leigh system, based on one million simulated replications.  The mean number is about $-1.351$.}
\end{figure}

\begin{table}[H]
\caption{Various statistics of the Leigh system over eight days with 360 coups per day (from a simulation with one million replications), and the values observed by Leigh's team in 1966.  N.A. means ``not available.''\label{tab:Leigh-results}}
\catcode`@=\active\def@{\phantom{0}}
\tabcolsep=0.5mm
\begin{center}
\begin{tabular}{lcc}
\hline\noalign{\smallskip}
statistic & simulation & Leigh team's \\
          & estimate   & observation  \\
\noalign{\smallskip}\hline\noalign{\smallskip}
number of winning progressions & 1.484665 & 27 \\
amount won per winning progression & 8236.074805 & 29,602.1 \\
total amount won from winning progressions & 12,227.812000 & 799,258 \\
number of losing progressions & 2055.310293 & N.A. \\
amount lost per losing progression & 9.975479 & N.A. \\
total amount lost from losing progressions & 20,502.704884  & N.A. \\
number of incomplete progressions & 42.214352 & N.A. \\
amount won per incomplete progression & 63.766071 & N.A. \\
total amount won from incomplete progressions & 2691.843352 & N.A. \\
total amount bet & 413,287.742596 & N.A. \\
\noalign{\smallskip}\hline
\end{tabular}
\end{center}
\end{table}

Another benefit of the simulation is that it allows us to estimate the distribution of the number $N$ of winning progressions in an eight-day session (with 360 coups per day) of the Leigh system.  We notice that $N$ is approximately Poisson distributed.  This is made plausible by \citeauthor{A89}'s (\citeyear{A89}) Poisson clumping heuristic.  Of course, it is not exactly Poisson because it has finite support.\footnote{We conjecture that the maximum possible number of winning progressions in eight 360-coup days is 144.  A winning progression can be achieved in 58 coups with the sequence $(\text{WWL})^{19}\text{W}$ of wins and losses.  The initial list $(a_1,a_2,a_3,a_4)=(1,2,3,4)$ becomes $(a_{20},a_{21},a_{22},a_{23},a_{24})=(657,907,1252,1728,2385)$, resulting in a win of 6919.  Here the sequence $(a_n)_{n\ge1}$ is Fibonacci-like and given by $a_n=a_{n-1}+a_{n-4}$ for each $n\ge5$.  This is sequence number \href{https://oeis.org/A003269}{A003269} (shifted) in the \textit{Online Encyclopedia of Integer Sequences}.  Six successive such progressions fit in a 360-coup day, hence 48 in an eight-day session.  Finally, this conclusion applies to any three of the six even chances.  For example, red, odd, and low win if one of the numbers 1, 3, 5, 7, and 9 appears, and lose if one of the numbers 20, 22, 24, 26, and 28 appears.  Thus, we have $48\times3=144$ winning progressions, each winning 6919, for a total win of 996,336.  This is the analogue of Leigh's 800,000 --- we have not accounted for the losing or incomplete progressions, which in this contrived scenario contribute to total profit in a relatively insignificant way.}  But a Poisson random variable $N_0$, with parameter equal to the sample mean $\mu_0=1.484665$, does appear to be a reasonable approximation (see Table~\ref{tab:distrib-of-no-of-wins}), even though, because of the large sample size ($n=1{,}000{,}000$), the sample does not come close to passing a chi-squared goodness-of-fit test.

\begin{table}[htb]
\caption{Distribution of the number $N$ of winning progressions in an eight-day session (with 360 coups per day) of the Leigh system, estimated from a simulation with one million replications.  The Poisson random variable $N_0$ with parameter $\mu_0=1.484665$ is a first approximation to $N$, and the Poisson random variable $N_1$ with parameter $\mu_1=1.51$ is plausibly a stochastic upper bound for $N$. \label{tab:distrib-of-no-of-wins}}
\catcode`@=\active\def@{\phantom{0}}
\tabcolsep=1.5mm
\begin{center}
\begin{tabular}{ccccccccc}
\hline\noalign{\smallskip}
$n$ & simulation & standard & Poisson &@& simulation & standard & Poisson \\
    & estimate of & error of & approx. && estimate of & error of & bound \\
    & $\P(N=n)$ & estimate & $\P(N_0=n)$ && $\P(N\ge n)$ & estimate & $\P(N_1\ge n)$ \\
\noalign{\smallskip}\hline\noalign{\smallskip}
0 & 0.223507 & 0.000417 &  0.226578  && 1.000000 & 0.000000 & 1.000000 \\
1 & 0.337618 & 0.000473 &  0.336393  && 0.776493 & 0.000417 & 0.779090 \\
2 & 0.253068 & 0.000435 &  0.249715  && 0.438875 & 0.000496 & 0.445516 \\
3 & 0.124061 & 0.000330 &  0.123581  && 0.185807 & 0.000389 & 0.193668 \\
4 & 0.044806 & 0.000207 &  0.045869  && 0.061746 & 0.000241 & 0.066904 \\
5 & 0.013036 & 0.000113 &  0.013620  && 0.016940 & 0.000129 & 0.019051 \\
6 & 0.003160 & 0.000056 &  0.003370  && 0.003904 & 0.000062 & 0.004599 \\
7 & 0.000608 & 0.000025 &  0.000715  && 0.000744 & 0.000027 & 0.000962 \\
8 & 0.000116 & 0.000011 &  0.000132  && 0.000136 & 0.000012 & 0.000177 \\
9 & 0.000020 & 0.000004 &  0.000022  && 0.000020 & 0.000004 & 0.000029 \\
\noalign{\smallskip}\hline
\end{tabular}
\end{center}
\end{table}

But instead of trying to find a Poisson random variable that provides a good fit, it will be more useful to find one that provides a stochastic upper bound.  That is, we want a Poisson random variable $N_1$ such that
\[
\P(N\ge n)\le \P(N_1\ge n),\qquad n\ge0.
\]
This can be achieved with parameter $\mu_1$ a bit larger than $\mu_0$, say $\mu_1=1.51$.  Table~\ref{tab:distrib-of-no-of-wins} confirms the stochastic bound, with at least two standard errors to spare for each $n$.  This leads to the (plausible but not proved) conclusion that
\[
\P(N\ge27)\le\P(N_1\ge27)<1.5\times10^{-24}.
\]
We can therefore rule out the possibility that the Leigh team's success occurred as reported and by chance.  We conclude that Leigh's book is a work of fiction.  A very entertaining book to be sure, but a fictional one.  

Our conclusion matches that of \cite{K18}, who used a different argument.  He simulated the Leigh system for 1.5 billion spins and found that the ratio of total profit to total amount bet was very close to its theoretical limit, just as we did in \eqref{consistency}. In effect, he used Monte Carlo simulation to show that the Leigh system cannot be profitable in the long term, which we deduced more reliably from Corollary~\ref{cor:multiple_bets}.  But Kendall did not seriously address the question of whether the Leigh team's success could have been the result of chance.  He simulated a day's play (360 spins) 30 times (not 30,000 or 30 million, but 30).  Also, Kendall ignored the \textit{en prison} rule and assumed the house maximum to be 4000 francs instead of Casino Municipale's 2600 francs.  (The number 4000 was the house maximum, in shillings, at London's Regency.)

\section{Compound games with one bet per coup}\label{sec:compound-one}

Here we would like to make one further generalization, allowing for compound games.  A \textit{compound game} (in the context of gambling, not game theory) is a game that allows additional bets, proportional to the original one, during the course of the game.  The best-known examples are doubling and splitting in blackjack, free odds bets in craps, and various proprietary poker-like games in which there is an ante bet and then options for additional bets under certain conditions.  One of the simpler examples of the latter is Three Card Poker, which has been a popular game since the early 1990s.

\begin{example}[Three Card Poker]\label{example:poker} Three Card Poker is played with a standard 52-card deck, which is reshuffled between coups.  Here we consider only the principal wager, the so-called ante-play wager.  Each gambler makes an \textit{ante wager} and receives three cards face down. The dealer also receives three cards face down.  Each gambler, after examining his hand, must decide whether to fold or play.  If he folds, he loses his ante wager.  If he plays, he must make an additional wager, called a \textit{play wager}, equal in size to the ante wager.  The dealer then reveals his hand to determine whether it \textit{qualifies} with a rank of queen high or better.  If the dealer's hand fails to qualify, the gambler's ante wager is paid even money and his play wager is pushed.  If the dealer's hand qualifies, it is compared with the gambler's hand.  If the gambler's hand outranks the dealer's hand, both of the gambler's wagers are paid even money.  If the dealer's hand outranks the gambler's hand, both of the gambler's wagers are lost.   If the gambler's hand and the dealer's hand are equal in rank, both of the gambler's wagers are pushed.

The ranking of hands in Three Card Poker is straight flush, three of a kind, straight, flush, one pair, and no pair.  There is one remaining detail called the \textit{ante bonus}.  Regardless of whether the gambler plays, regardless of whether the dealer qualifies, and regardless of the outcome of the ante-play wager, if the gambler's hand is a straight or better, he receives a bonus payoff on his ante wager.  This bonus amounts to 1 to 1 for a straight, 4 to 1 for three of a kind, and 5 to 1 for a straight flush.  This completes our description of the game.  See \citeauthor{E10} (\citeyear{E10}, Section~16.2) for more detail.

It can be shown that it is optimal to play with unsuited Q-6-4 or better (that is, these are the hands for which the play wager has nonnegative conditional expectation).  If the gambler adopts this strategy, a one-unit initial bet results in a bet of size
\[
B_1^\circ=\begin{cases}1&\text{with probability $\dfrac{72}{221}$},\\ 2&\text{otherwise},\end{cases}
\]
with expectation
\[
\E[B_1^\circ]=\frac{370}{221}\approx1.674208.
\]
Furthermore, the corresponding expected profit to the gambler is
\[
\E[X_1^\circ]=-\frac{35{,}253{,}012}{\binom{52}{3}\binom{49}{3}}+\frac{1168}{\binom{52}{3}}=-\frac{686{,}689}{20{,}358{,}520}\approx-0.033730.
\]
One can therefore define the house advantage as
\[
\HA(B_1^\circ,X_1^\circ)=-\frac{\E[X_1^\circ]}{\E[B_1^\circ]}=\frac{686{,}689}{34{,}084{,}400}\approx0.020147.
\]
\end{example}

To account for such games, we replace the wager $(1,R_n^\circ,X_n^\circ)$, in which exactly one unit is bet, by a compound wager in which the base bet is one unit and, under certain conditions, that bet is incremented during the $n$th coup, for a total bet of $B_n^\circ$.
We let $B_n^*$ denote the base bet at the $n$th coup, and therefore assume
\begin{equation*}
B_n = B_n^* B_n^\circ, \qquad
R_n = B_n^* R_n^\circ, \qquad
X_n = B_n^* X_n^\circ,
\end{equation*}
where $R_n^\circ$ and $X_n^\circ$ represent, respectively, the return and the profit obtained when one places an initial one-unit bet and then increments this bet during the coup.

In this compound setting,
\begin{itemize}
    \item the sequence $(B_n^*)_{n\ge1}$ is predictable with respect to $(\F_n)_{n\ge0}$;
    \item the sequence $(B_n^\circ,R_n^\circ)_{n\ge1}$ is i.i.d.\ and adapted to $(\F_n)_{n\ge0}$, and $(B_n^\circ,R_n^\circ)$ is independent of $\F_{n-1}$ for each $n\ge1$;
     \item the sequence $(B_n^\circ,X_n^\circ)_{n\ge1}$ is i.i.d.\ and adapted to $(\F_n)_{n\ge0}$, and $(B_n^\circ,X_n^\circ)$ is independent of $\F_{n-1}$ for each $n\ge1$.
\end{itemize}

Under these assumptions,
\begin{align*}
\frac{\sum_{k=1}^n \E[R_k\mid\F_{k-1}]}{\sum_{k=1}^n \E[B_k\mid\F_{k-1}]}&=\frac{\sum_{k=1}^n B_k^*\,\E[R_1^\circ]}{\sum_{k=1}^n B_k^*\,\E[B_1^\circ]}=\frac{\E[R_1^\circ]}{\E[B_1^\circ]}=\RTP(B_1^\circ,R_1^\circ),\\
\frac{\sum_{k=1}^n \E[X_k\mid\F_{k-1}]}{\sum_{k=1}^n \E[B_k\mid\F_{k-1}]}&=\frac{\sum_{k=1}^n B_k^*\,\E[X_1^\circ]}{\sum_{k=1}^n B_k^*\,\E[B_1^\circ]}=\frac{\E[X_1^\circ]}{\E[B_1^\circ]}=-\HA(B_1^\circ,X_1^\circ).
\end{align*}

\begin{corollary}\label{cor:one_bet_compound}
Under the above assumptions, suppose also that 
\[
\sup_{n \ge 1} \E[(B_n^*)^2] < \infty,\qquad\liminf_{n\to\infty}\frac{1}{n}\sum_{k=1}^n B_k^*>0\quad\text{a.s.}, \qquad\text{and} \qquad \E[(B_1^\circ)^2]<\infty.
\]
If $\E[(R_1^\circ)^2] < \infty$, or equivalently $\E[(X_1^\circ)^2] < \infty$, then
    \[
    \lim_{n\to\infty}\mathfrak{R}_n=\RTP(B_1^\circ,R_1^\circ)\quad\text{a.s.} \qquad \text{and} \qquad \lim_{n\to\infty}\mathfrak{X}_n=-\HA(B_1^\circ,X_1^\circ)\quad\text{a.s.}
        \]
\end{corollary}

\begin{remark}
For compound games, there is some controversy concerning the definitions of the RTP and the house advantage.  
One approach considers the \textit{total bet} placed during a round, while another evaluates these quantities only with respect to the \textit{base bet}.  
This distinction is well known in the literature, and explained, for example, in~\citeauthor{HC05} (\citeyear{HC05}, p.~56).
The former seems preferable from the gambler's perspective because it gives him credit for the additional risk he takes, but the latter may be preferable from the casino's perspective for several reasons.

For example, in Three Card Poker (Example~\ref{example:poker}), one could also compute the house advantage as 
\[
\HA(1,X_1^\circ)=-\E[X_1^\circ]=\frac{686{,}689}{20{,}358{,}520}\approx0.033730.
\]

The results in this section and in the next one are stated using the total-bet convention.  
However, the corresponding statements for the base-bet convention follow immediately: one simply replaces the total bet $B_{n}^{\circ}$ corresponding to a base bet $1$ with simply $1$. All the assertions then remain valid under this alternative interpretation.
\end{remark}

\begin{example}[Craps: pass-line bet with 3/4/5-times odds]\label{example:craps} If we restrict the craps gambler to pass-line bets with 3/4/5-times odds, the present framework is sufficient.  

We begin with a brief description of these bets.  Craps is played by rolling a pair of fair dice repeatedly.  With a few exceptions, only the total of the two dice matters, and its distribution is
\[
\pi_k:=\frac{6-|7-k|}{36},\quad k\in\{2,3,4,\ldots,12\},
\]
so a sequence of dice rolls is an i.i.d.\ sequence $(T_n)_{n\ge1}$ with common distribution $\pi$. 
But here we index the filtration not by dice rolls but by pass-line decisions.  The first roll is the \textit{come-out roll}.  If $T_1\in\{7,11\}$ (a \textit{natural}), the pass-line bet is won, and if $T_1\in\{2,3,12\}$ (a \textit{craps number}), the pass-line bet is lost.  Otherwise, $T_1\in \{4,5,6,8,9,10\}$, and the gambler is said to have established the \textit{point} $T_1$.  The pass-line bet requires a geometrically distributed number of additional rolls for its resolution.  If the point $T_1$ reappears before a 7, the bets is won.  If a 7 appears before the point $T_1$ reappears, the bet is lost.  More precisely, with
\[
N:=\begin{cases}1&\text{if $T_1\in\{2,3,7,11,12\}$},\\
\min\{n\ge2:T_n=T_1\text{ or }T_n=7\}&\text{if $T_1\in\{4,5,6,8,9,10\}$}.\end{cases}
\]
the pass-line bet is resolved at the $N$th roll, and the process begins again.

What makes this a compound bet is the fact that, once a point is established, the gambler is permitted to make an additional bet, equal to 3 (resp., 4, 5) times the amount of his pass-line bet if the point is 4 or 10 (resp., 5 or 9, 6 or 8), that the pass-line bet will be won, and this bet pays fair odds, thus is often described as ``free.''  If the point is $k\in\{4,5,6,8,9,10\}$, the conditional probability of winning is $\pi_k/(\pi_k+\pi_7)$ and that of losing is $\pi_7/(\pi_k+\pi_7)$, hence the payoff odds are $\pi_7$ to $\pi_k$, or 2 to 1 if the point is 4 or 10, 3 to 2 of the point is 5 or 9, and 6 to 5 if the point is 6 or 8.

Therefore the total amount bet, if the pass-line bet is one unit, has the form
\[
B_1^\circ:=\begin{cases}1&\text{if $T_1\in\{2,3,7,11,12\}$},\\
                       1+3&\text{if $T_1\in\{4,10\}$},\\
                       1+4&\text{if $T_1\in\{5,9\}$},\\
                       1+5&\text{if $T_1\in\{6,8\}$},\end{cases}
\]
and its mean is
\begin{align*}
\E[B_1^\circ]&=1(\pi_2+\pi_3+\pi_7+\pi_{11}+\pi_{12})+4(\pi_4+\pi_{10})+5(\pi_5+\pi_9)+6(\pi_6+\pi_8)\\
&=34/9.
\end{align*}
The total profit from such a bet has the form
\[
X_1^\circ:=\begin{cases}1&\text{if $T_1\in\{7,11\}$},\\
                       -1&\text{if $T_1\in\{2,3,12\}$},\\
                       1+6&\text{if $T_1\in\{4,5,6,8,9,10\}$ and $T_N=T_1$},\\
                       -(1+3)&\text{if $T_1\in\{4,10\}$ and $T_N=7$},\\
                       -(1+4)&\text{if $T_1\in\{5,9\}$ and $T_N=7$},\\
                       -(1+5)&\text{if $T_1\in\{6,8\}$ and $T_N=7$}.\end{cases}
\]
and its mean is
\[
\E[X_1^\circ]=2\bigg[\pi_7+\pi_{11}+\sum_{k\in\{4,5,6,8,9,10\}}\pi_k\,\frac{\pi_k}{\pi_k+\pi_7}\bigg]-1=-\frac{7}{495}
\]
because the free-odds portion of the bet is fair.

Thus, we find that
\[
\HA(B_1^\circ,X_1^\circ)=\frac{7/495}{34/9}=\frac{7}{1870}\approx0.0037433,
\]
and Corollary~\ref{cor:one_bet_compound} applies.

It is important to realize that, if we wanted to model a more elaborate craps system, for example one that includes pass-line bets and come bets, together with free-odds bets, then we would need a filtration indexed by dice rolls, not pass-line decisions, and in that case future-dependence becomes a significant issue.  We address this in Section~\ref{sec:future-dependence}.
\end{example}

\begin{remark}
Under the assumptions of this section (and of the following), requiring the sequences $(B_n^\circ,R_n^\circ)_{n\ge1}$ and $(B_n^\circ,X_n^\circ)_{n\ge1}$ to be i.i.d.\ means that the gambler may vary the \textit{entry} bet, but must follow a fixed behavior \textit{within} each coup.  
For instance, in craps the gambler must consistently adopt the same rule for free odds bets: always take the maximum, never take them, or decide according to some fixed random mechanism (such as tossing a coin).  
In addition, this intra-coup behavior must be independent of the choice of the initial wager.

This assumption is, of course, not always realistic.  
When it fails, Corollary~\ref{cor:one_bet_compound} no longer applies directly, but one may instead appeal to Theorem~\ref{theorem:main} to obtain analogous conclusions.
\end{remark}

\section{Compound games with multiple bets per coup}\label{sec:compound-multiple}

We want to generalize the preceding section, allowing multiple bets at each coup.  For example, think of six Three Card Poker gamblers, each playing optimally against the same dealer hand but each with his own hand and his own betting system.

Assuming $d$ betting opportunities, we introduce the random vectors
\begin{align*}
\bm B_n^\circ=(B_{n,1}^\circ,\ldots,B_{n,d}^\circ),\qquad \bm B_n^*=(B_{n,1}^*,\ldots,B_{n,d}^*),\\
\bm R_n^\circ=(R_{n,1}^\circ,\ldots,R_{n,d}^\circ),\qquad \bm X_n^\circ=(X_{n,1}^\circ,\ldots,X_{n,d}^\circ),
\end{align*}
where $B_{n,j}^\circ$ represents a one-unit base bet plus any subsequent bet made on the $j$th betting opportunity at the $n$th coup, $R_{n,j}^\circ$ represents the return from the bet of $B_{n,j}^\circ$ on the $j$th betting opportunity at the $n$th coup, $X_{n,j}^\circ$ represents the profit from the bet of $B_{n,j}^\circ$ on the $j$th betting opportunity at the $n$th coup, and $B_{n,j}^*$ represents the amount of the base bet on the $j$th betting opportunity at the $n$th coup.  Then
\begin{equation*}
B_n=\bm B_n^*\bm\cdot\bm B_n^\circ,\quad R_n=\bm B_n^*\bm\cdot\bm R_n^\circ,\quad X_n=\bm B_n^*\bm\cdot\bm X_n^\circ.
\end{equation*}

In the compound setting, 
\begin{itemize}
    \item the sequence $(\bm B_n^*)_{n\ge1}$ is predictable with respect to $(\F_n)_{n\ge0}$;
    \item the sequence $(\bm B_n^\circ,\bm R_n^\circ)_{n\ge1}$ is i.i.d.\ and adapted to $(\F_n)_{n\ge0}$, and $(\bm B_n^\circ,\bm R_n^\circ)$ is independent of $\F_{n-1}$ for each $n\ge1$;
    \item the sequence $(\bm B_n^\circ,\bm X_n^\circ)_{n\ge1}$ is i.i.d.\ and adapted to $(\F_n)_{n\ge0}$, and $(\bm B_n^\circ,\bm X_n^\circ)$ is independent of $\F_{n-1}$ for each $n\ge1$.
\end{itemize}

Under these assumptions,
\begin{align*}
\frac{\sum_{k=1}^n \E[R_k\mid\F_{k-1}]}{\sum_{k=1}^n \E[B_k\mid\F_{k-1}]}&=\frac{\sum_{k=1}^n \bm B_k^*\bm\cdot\E[\bm R_1^\circ]}{\sum_{k=1}^n \bm B_k^*\bm\cdot\E[\bm B_1^\circ]},\\
\frac{\sum_{k=1}^n \E[X_k\mid\F_{k-1}]}{\sum_{k=1}^n \E[B_k\mid\F_{k-1}]}&=\frac{\sum_{k=1}^n \bm B_k^*\bm\cdot\E[\bm X_1^\circ]}{\sum_{k=1}^n \bm B_k^*\bm\cdot\E[\bm B_1^\circ]}.
\end{align*}

\begin{corollary}\label{cor:multiple_bets-compound}
Under the above assumptions, suppose also that 
\[
\sup_{n \ge 1} \E[|\bm B_n^*|^2] < \infty,\qquad\liminf_{n\to\infty}\frac{1}{n}\sum_{k=1}^n \bm B_k^*\bm\cdot\bm B_k^\circ>0\quad\text{a.s.}, \quad \text{and} \quad \E[|\bm B_1^\circ|^2]<\infty.
\]
If $\E[|\bm R_1^\circ|^2] < \infty$, or equivalently $\E[|\bm X_1^\circ|^2] < \infty$, then 
\begin{align*}
&\min_{1\le i \le d}\RTP(B_{1,i}^\circ,R_{1,i}^\circ)\le\liminf_{n\to\infty}\frac{\sum_{k=1}^n \bm B_k^*\bm\cdot\E[\bm R_1^\circ]}{\sum_{k=1}^n \bm B_k^*\bm\cdot\E[\bm B_1^\circ]}=\liminf_{n\to\infty}\mathfrak{R}_n\\
&\qquad{}\le\limsup_{n\to\infty}\mathfrak{R}_n=\frac{\sum_{k=1}^n \bm B_k^*\bm\cdot\E[\bm R_1^\circ]}{\sum_{k=1}^n \bm B_k^*\bm\cdot\E[\bm B_1^\circ]}\le\max_{1\le i \le d}\RTP(B_{1,i}^\circ,R_{1,i}^\circ)\quad\text{a.s.}
\end{align*}
and
\begin{align*}
&-\max_{1\le i \le d}\HA(B_{1,i}^\circ,X_{1,i}^\circ)\le\liminf_{n\to\infty}\frac{\sum_{k=1}^n \bm B_k^*\bm\cdot\E[\bm X_1^\circ]}{\sum_{k=1}^n \bm B_k^*\bm\cdot\E[\bm B_1^\circ]}=\liminf_{n\to\infty}\mathfrak{X}_n\\
&\qquad{}\le\limsup_{n\to\infty}\mathfrak{X}_n=\frac{\sum_{k=1}^n \bm B_k^*\bm\cdot\E[\bm X_1^\circ]}{\sum_{k=1}^n \bm B_k^*\bm\cdot\E[\bm B_1^\circ]}\le-\min_{1\le i \le d}\HA(B_{1,i}^\circ,X_{1,i}^\circ)\quad\text{a.s.}
\end{align*}

The limits exist under either of two additional conditions: 
\begin{itemize}
    \item If $\RTP(B_{1,1},R_{1,1}^\circ)=\cdots=\RTP(B_{1,d},R_{1,d}^\circ)$, or equivalently $\HA(B_{1,1}^\circ,\break X_{1,1}^\circ)=\cdots=\HA(B_{1,d}^\circ,X_{1,d}^\circ)$, then
\[
\lim_{n\to\infty}\mathfrak{R}_n=\RTP(B_{1,1}^\circ,R_{1,1}^\circ)\quad\text{a.s.} \]
and \[\lim_{n\to\infty}\mathfrak{X}_n=-\HA(B_{1,1}^\circ,X_{1,1}^\circ)\quad\text{a.s.}
\]

\item If $\lim_{n\to\infty}\sum_{k=1}^n \bm B_k^*/\sum_{k=1}^n \bm B_k^*\bm\cdot\bm1=\bm w$ a.s., then
\[
\lim_{n\to\infty}\mathfrak{R}_n=\sum_{i=1}^d w_i\,\E[B_{1,i}^\circ]\,\RTP(B_{1,i}^\circ,R_{1,i}^\circ)\bigg/\sum_{i=1}^d w_i\,\E[B_{1,i}^\circ]\quad\text{a.s.}
\]
and 
\[
\lim_{n\to\infty}\mathfrak{X}_n=-\sum_{i=1}^d w_i\,\E[B_{1,i}^\circ]\,\HA(B_{1,i}^\circ,X_{1,i}^\circ)\bigg/\sum_{i=1}^d w_i\,\E[B_{1,i}^\circ]\quad\text{a.s.}
\]
\end{itemize}

\end{corollary}

\section{Games with future dependence}\label{sec:future-dependence}

In several casino games, gamblers may place new wagers while earlier ones are still unresolved.  
Typical examples include aforementioned roulette with the \textit{en prison} rule and craps.  
In such settings, the sequences $(B_n)_{n\ge1}$, $(R_n)_{n\ge1}$, and $(X_n)_{n\ge1}$ need not be adapted to the natural filtration indexed by spins of the wheel or rolls of the dice; the natural unit of evolution is therefore a full \textit{round} rather than a single coup or spin or roll.

The guiding idea of this section is therefore to work with a filtration $(\F_n)_{n\ge0}$ which encodes the randomness of the game at each spin or roll, but where bets, returns, and profit may be revealed only at the end of the round containing them.  
We assume the existence of a natural subdivision into \textit{rounds}, whose endpoints are stopping times, so that all pending amounts are fully resolved by the end of each round.  
Thus the analysis moves from being ``spin- or roll-based'' to ``round-based.''

More precisely, consider a strictly increasing sequence of almost surely finite stopping times $(\tau_m)_{m\ge0}$ with $\tau_0=0$.  
The $m$th round consists of the spins or rolls $\tau_{m-1}+1,\dots,\tau_m$, and the $\sigma$-algebra corresponding to the completion of this round is given by
\[
\G_m := \F_{\tau_m}, \qquad m\ge0.
\]
Let $L_m=\tau_m-\tau_{m-1}$ denote the random length of the $m$th round, and assume
\[
\sup_{m\ge1}\E[L_m^2]<\infty.
\]

\begin{theorem}\label{thm:round}
Under the above setting, let $(B_n,R_n,X_n)_{n\ge1}$ be a sequence of wagers with uniformly bounded random variables.  
Assume that there exist constants $\rho'$ and $\rho''$ such that
\[
\rho' \le \frac{\E[R_n\mid\F_{n-1}]}{\E[B_n\mid\F_{n-1}]} \le \rho'' \quad \text{a.s.},
\]
or equivalently constants $\chi'$ and $\chi''$ such that
\[
\chi' \le \frac{\E[X_n\mid\F_{n-1}]}{\E[B_n\mid\F_{n-1}]} \le \chi'' \quad \text{a.s.}
\]
for all $n$.
Define the round totals
\[
\overline{B}_m:=\sum_{i=\tau_{m-1}+1}^{\tau_m} B_i,\quad
\overline{R}_m:=\sum_{i=\tau_{m-1}+1}^{\tau_m} R_i,\quad
\overline{X}_m:=\sum_{i=\tau_{m-1}+1}^{\tau_m} X_i.
\]
Assume further that these are $\G_m$-measurable and that
\[
\liminf_{m\to\infty} \frac{1}{m}\sum_{j=1}^m \overline{B}_j > 0 \quad \text{a.s.}
\]
Then, introducing the notation
\[
\overline{\mathfrak{R}}_m:=\frac{\sum_{j=1}^m \overline{R}_j}{\sum_{j=1}^m \overline{B}_j}\quad\text{and}\quad
\overline{\mathfrak{X}}_m:=\frac{\sum_{j=1}^m \overline{X}_j}{\sum_{j=1}^m \overline{B}_j},
\]
we have
\[
\rho' \le \liminf_{m\to\infty} \overline{\mathfrak{R}}_m
\le \limsup_{m\to\infty} \overline{\mathfrak{R}}_m
\le \rho''\quad \text{a.s.}
\]
and 
\[
\chi'  \le\ \liminf_{m\to\infty} \overline{\mathfrak{X}}_m
\le \limsup_{m\to\infty} \overline{\mathfrak{X}}_m
\le \chi''\quad \text{a.s.}
\]
In particular, if $\rho'=\rho''=\rho$, or equivalently $\chi'=\chi''=\chi$, then
\[
\lim_{m\to\infty}\overline{\mathfrak{R}}_m = \rho \quad \text{a.s.} \qquad \text{and}
\qquad
\lim_{m\to\infty} \overline{\mathfrak{X}}_m = \chi
\quad \text{a.s.}
\]
\end{theorem}

\begin{proof}
We apply Theorem~\ref{theorem:main} to $(\G_m)$ and $(\overline{B}_m,\overline{R}_m,\overline{X}_m)$.  
Since $\overline{R}_m\le rL_m$ and $\sup_{m\ge1}\E[L_m^2]<\infty$, we have $\sup_{m\ge1}\E[\overline{R}_m^2]<\infty$, and similarly for $\overline{B}_m$.  
Write
\[
\overline{B}_m=\sum_{j>\tau_{m-1}} B_j\,\mathbf 1_{\{\tau_m>j-1\}}, \qquad \overline{R}_m=\sum_{j>\tau_{m-1}} R_j\,\mathbf 1_{\{\tau_m>j-1\}}.
\]
We compute, by the tower property and the fact that $\mathbf 1_{\{\tau_m>j-1\}}$ is $\F_{j-1}$-measurable, the following bound:
\begin{align*}
    \E[\overline{R}_m\mid \G_{m-1}]
    &=\sum_{j > \tau_{m-1}+1}\E[\mathbf 1_{\{\tau_m>j-1\}}\,\E[R_j\mid\F_{j-1}]\mid \G_{m-1}]\\
    &\le \sum_{j > \tau_{m-1}+1}\E[\mathbf 1_{\{\tau_m>j-1\}}\,\E[B_j\mid\F_{j-1}]\, \rho''\mid \G_{m-1}]\\
    &=\E[\overline{B}_m\mid \G_{m-1}]\, \rho''.
\end{align*}
We find therefore that
\[
\frac{\E[\overline{R}_m\mid\G_{m-1}]}{\E[\overline{B}_m\mid\G_{m-1}]} \le \rho''.
\]
The lower bound is similar, and the same applies to $(\overline{X}_m)_{m\ge 1}$.  
The theorem now follows.
\end{proof}

\begin{remark}
The previous theorem is formulated in terms of the round-based quantities $\overline{\mathfrak{R}}_n$ and $\overline{\mathfrak{X}}_n$, reflecting the natural decomposition of the game into rounds.
However, since the round lengths have uniformly bounded second moments, the discrepancy between round-based and spin- or roll-based totals can be proved to be negligible in the long term.
Consequently, the same asymptotic bounds hold when $\overline{\mathfrak{R}}_n$ and $\overline{\mathfrak{X}}_n$ are replaced by the usual ratios $\mathfrak{R}_n$ and $\mathfrak{X}_n$.
\end{remark}

\begin{example}[Roulette with \textit{en prison}]
We revisit roulette with the \textit{en prison} rule.  
Now gamblers may also place bets during the spins required to resolve the even chances following a zero outcome, so the filtration is necessarily indexed by the individual spins.  Rounds are what we referred to as coups in Section~\ref{sec:Leigh}.  Round lengths are geometrically distributed with parameter $36/37$ and therefore have finite second moment.
\end{example}

\begin{example}[Craps]
In Example~\ref{example:craps} we indexed the filtration by pass-line decisions.  
If we also allow come bets (with $3/4/5$-times odds), which are mathematically equivalent to the pass-line but initiated on rolls other than come-out rolls and resolved possibly before or after the pending pass-line bet, then the filtration must instead be indexed by the dice rolls themselves.  
In this formulation, $B_n$ and $X_n$ at roll $n$ need not be $\F_n$-measurable, but the conditional expectation formula remains unchanged:
\[
\frac{\E[X_n\mid\F_{n-1}]}{\E[B_n\mid\F_{n-1}]} 
= \frac{\E[X_n^{\circ}\mid\F_{n-1}]}{\E[B_n^{\circ}\mid\F_{n-1}]} 
= -\HA(B_n^{\circ},X_n^{\circ}) 
= -\frac{7}{1870}.
\]

When a pass-line bet is lost by rolling a 7, that event is called a \textit{seven-out}.
In craps, a natural division into rounds is given by the seven-out events, in which a new shooter is selected, and each pending wager is resolved.
Specifically, $(\tau_m)_{m\ge 0}$ is defined as follows:
\begin{equation*}
    \tau_0 := 0, \quad \tau_m := \inf\{n> \tau_{m-1}: \text{the shooter sevens out on the $n$th roll}\}.
\end{equation*}
The random variables $L_m := \tau_m - \tau_{m-1}$ are i.i.d., and their description is given in \citeauthor{E10} (\citeyear{E10}, Section 15.2); in particular their common second moment is finite, and this is enough to apply Theorem~\ref{thm:round} to derive the following: 
\[
\lim_{m\to\infty}\overline{\mathfrak{X}}_m=-\frac{7}{1870},
\]
even though wagers may require additional rolls before being fully resolved.
\end{example}

\begin{example}[Slot machines with accumulation features]
Consider a slot machine where part of each spin’s return is accumulated into a bonus pool paid out every $s$ spins.  
Each bet $B_n$ is immediately measurable, while each return $R_n$ consists of an instantly paid portion plus a deferred portion released during the bonus round. This setting suggests a natural division into rounds of fixed length $L_m=s$, and Theorem~\ref{thm:round} applies directly.
\end{example}

\end{document}